\documentclass[a4paper]{amsart}
\usepackage{amsmath}
\usepackage{amssymb}
\usepackage{amscd}
\usepackage{braket}

\newtheorem{thm}{Theorem}[section]

\theoremstyle{definition}
\newtheorem{algorithm}[thm]{Algorithm}

\newtheorem{exmp}[thm]{Example}

\theoremstyle{remark}
\newtheorem{rem}[thm]{Remark}

\newcommand\ZZ{{\mathbb Z}}

\newcommand\Ob{\operatorname{ob}}
\newcommand\Mor{\operatorname{mor}}

\newcommand\id{\operatorname{id}}

\newcommand\Der{\operatorname{Der}}
\newcommand\Ider{\operatorname{Ider}}

\newcommand\C{\mathcal{C}}

\newcommand\F{\mathcal{F}}

\newcommand\Mod{\operatorname{Mod}}

\newcommand\Hlgy{\operatorname{H}}
\newcommand\HBW{\Hlgy_{BW}}

\newcommand{\figVII}{
\begin{picture}(160,110)
 \thicklines
 \qbezier[30](32,60)(90,-35)(148,60)
 \qbezier[30](24,50)(90,-45)(156,50)
 \thinlines
 \put(60,100){\makebox(0,0){\normalsize$\bullet$}}
 \put(60,110){\makebox(0,0)[t]{\normalsize$n$}}
 \put(30,60){\makebox(0,0){\normalsize$\bullet$}}
 \put(16,60){\makebox(0,0){\normalsize$n-1$}}
 \put(120,100){\makebox(0,0){\normalsize$\bullet$}}
 \put(120,110){\makebox(0,0)[t]{\normalsize$1$}}
 \put(150,60){\makebox(0,0){\normalsize$\bullet$}}
 \put(155,60){\makebox(0,0){\normalsize$2$}}
 \put(152,69){\vector(-3,4){22}}
 \put(150,90){\makebox(0,0){\normalsize$\alpha_{1}$}}
 \put(112,105){\vector(-1,0){45}}
 \put(90,115){\makebox(0,0)[t]{\normalsize$\alpha_{n}$}}
 \put(52,101){\vector(-3,-4){22}}
 \put(25,90){\makebox(0,0){\normalsize$\alpha_{n-1}$}}
 \put(119,91){\vector(3,-4){22}}
 \put(120,70){\makebox(0,0){\normalsize$\beta_{1}$}}
 \put(68,95){\vector(1,0){45}}
 \put(90,90){\makebox(0,0)[t]{\normalsize$\beta_{n}$}}
 \put(40,65){\vector(3,4){22}}
 \put(65,75){\makebox(0,0){\normalsize$\beta_{n-1}$}}
\end{picture}
}

\newcommand{\figVI}{
\begin{picture}(160,110)
 \thicklines
 \qbezier[30](27,60)(90,-40)(153,60)
 \thinlines
 \put(60,100){\makebox(0,0){\normalsize$\bullet$}}
 \put(60,110){\makebox(0,0)[t]{\normalsize$n$}}
 \put(30,60){\makebox(0,0){\normalsize$\bullet$}}
 \put(16,60){\makebox(0,0){\normalsize$n-1$}}
 \put(120,100){\makebox(0,0){\normalsize$\bullet$}}
 \put(120,110){\makebox(0,0)[t]{\normalsize$1$}}
 \put(150,60){\makebox(0,0){\normalsize$\bullet$}}
 \put(155,60){\makebox(0,0){\normalsize$2$}}
 \put(150,60){\vector(-3,4){28}}
 \put(145,85){\makebox(0,0){\normalsize$\alpha_{1}$}}
 \put(120,100){\vector(-1,0){58}}
 \put(90,110){\makebox(0,0)[t]{\normalsize$\alpha_{n}$}}
 \put(60,100){\vector(-3,-4){28}}
 \put(30,85){\makebox(0,0){\normalsize$\alpha_{n-1}$}}
\end{picture}
}

\newcommand{\figIV}{
\begin{picture}(120,160)(-10,0)
 \put(80,140){\makebox(0,0){\normalsize$\bullet$}}
 \put(85,140){\makebox(0,0)[l]{\normalsize$y_1$}}
 \put(80,110){\makebox(0,0){\normalsize$\bullet$}}
 \put(85,110){\makebox(0,0)[l]{\normalsize$y_2$}}
 \put(80,80){\makebox(0,0){\normalsize$\bullet$}}
 \put(85,80){\makebox(0,0)[l]{\normalsize$y_3$}}
 \put(80,20){\makebox(0,0){\normalsize$\bullet$}}
 \put(85,20){\makebox(0,0)[l]{\normalsize$y_{n}$}}
 \put(80,140){\vector(-1,0){58}}
 \put(35,138){\makebox(0,0)[t]{\normalsize$\alpha_1$}}
 \put(80,110){\vector(-1,0){58}}
 \put(35,108){\makebox(0,0)[t]{\normalsize$\alpha_2$}}
 \put(80,80){\vector(-1,0){58}}
 \put(35,78){\makebox(0,0)[t]{\normalsize$\alpha_3$}}
 \put(80,20){\vector(-1,0){58}}
 \put(35,18){\makebox(0,0)[t]{\normalsize$\alpha_{n}$}}
 \put(20,140){\makebox(0,0){\normalsize$\bullet$}}
 \put(15,140){\makebox(0,0)[r]{\normalsize$x_1$}}
 \put(20,110){\makebox(0,0){\normalsize$\bullet$}}
 \put(15,110){\makebox(0,0)[r]{\normalsize$x_2$}}
 \put(20,80){\makebox(0,0){\normalsize$\bullet$}}
 \put(15,80){\makebox(0,0)[r]{\normalsize$x_3$}}
 \put(20,20){\makebox(0,0){\normalsize$\bullet$}}
 \put(15,20){\makebox(0,0)[r]{\normalsize$x_{n}$}}
 \put(80,140){\vector(-2,-1){58}}
 \put(65,133){\makebox(0,0)[tl]{\normalsize$\beta_2$}}
 \put(80,110){\vector(-2,-1){58}}
 \put(65,103){\makebox(0,0)[tl]{\normalsize$\beta_3$}}
 \put(80,80){\line(-2,-1){28}}
 \put(65,73){\makebox(0,0)[tl]{\normalsize$\beta_4$}}
 \put(80,20){\line(-2,-1){28}}
 \put(65,13){\makebox(0,0)[tl]{\normalsize$\beta_1$}}
 \put(40,30){\vector(-2,-1){18}}
 \put(40,150){\vector(-2,-1){18}}
 \thicklines
\qbezier[20](40,150)(80,170)(50,90)
\qbezier[20](52,6)(20,-10)(50,90)
 \thinlines
 \put(80,50){\makebox(0,0){\normalsize$\vdots$}}
 \put(20,50){\makebox(0,0){\normalsize$\vdots$}}
 \put(50,50){\makebox(0,0){\normalsize$\vdots$}}
\end{picture}
}

\newcommand{\figIII}{
\begin{picture}(110,130)
 \put(80,110){\makebox(0,0){\normalsize$\bullet$}}
 \put(85,110){\makebox(0,0)[l]{\normalsize$1$}}

 \put(80,80){\makebox(0,0){\normalsize$\bullet$}}
 \put(85,80){\makebox(0,0)[l]{\normalsize$2$}}
 \put(80,50){\makebox(0,0){\normalsize$\vdots$}}
 \put(80,20){\makebox(0,0){\normalsize$\bullet$}}
 \put(85,20){\makebox(0,0)[l]{\normalsize${n}$}}
 \put(20,80){\makebox(0,0){\normalsize$\bullet$}}
 \put(15,80){\makebox(0,0)[r]{\normalsize$x$}}
 \put(80,110){\vector(-2,-1){58}}
 \put(50,95){\makebox(0,0)[br]{\normalsize$\alpha_0$}}
 \put(80,80){\vector(-1,0){58}}
 \put(50,78){\makebox(0,0)[tr]{\normalsize$\alpha_1$}}
 \put(80,20){\vector(-1,1){58}}
 \put(50,50){\makebox(0,0)[tr]{\normalsize$\alpha_{n}$}}
\end{picture}
}

\newcommand{\figI}{
\begin{picture}(180,40)
 \put(10,20){\makebox(0,0){\normalsize$\bullet$}}
 \put(10,15){\makebox(0,0)[t]{\normalsize$1$}}
 \put(50,20){\vector(-1,0){38}}
 \put(30,25){\makebox(0,0)[b]{\normalsize$\alpha_{1}$}}
 \put(50,20){\makebox(0,0){\normalsize$\bullet$}}
 \put(50,15){\makebox(0,0)[t]{\normalsize${2}$}}
 \put(70,20){\makebox(0,0){\normalsize$\cdots$}}
 \put(90,20){\makebox(0,0){\normalsize$\bullet$}}
 \put(90,15){\makebox(0,0)[t]{\normalsize$n-2$}}
 \put(130,20){\vector(-1,0){38}}
 \put(110,25){\makebox(0,0)[b]{\normalsize$\alpha_{n-2}$}}
 \put(130,20){\makebox(0,0){\normalsize$\bullet$}}
 \put(130,15){\makebox(0,0)[t]{\normalsize$n-1$}}
 \put(170,20){\vector(-1,0){38}}
 \put(150,25){\makebox(0,0)[b]{\normalsize$\alpha_{n-1}$}}
 \put(170,20){\makebox(0,0){\normalsize$\bullet$}}
 \put(170,15){\makebox(0,0)[t]{\normalsize$n$}}
\end{picture}
}

\begin{document}
%

\title[On computation of the first Baues--Wirsching cohomology]{On computation of the first Baues--Wirsching cohomology of a freely-generated small category}
\author[Y. Momose]{MOMOSE, Yasuhiro}
\address[Y. Momose]{Department of Mathematical Sciences, Shinshu University, 3-1-1 Asahi, Matsumoto, Nagano 390-8621 JAPAN}
\email{momose@math.shinshu-u.ac.jp}

\author[Y. Numata]{NUMATA, Yasuhide}
\address[Y. Numata]{Department of Mathematical Sciences, Shinshu University, 3-1-1 Asahi, Matsumoto, Nagano 390-8621 JAPAN}
\email{nu@math.shinshu-u.ac.jp}

\begin{abstract}
The Baues--Wirsching cohomology is one of the cohomologies of a small category.
Our aim is to describe the first Baues--Wirsching cohomology of the small category generated by a finite quiver freely.
We consider the case where the coefficient is a natural system obtained by the composition of a functor and the target functor.
We give an algorithm to obtain generators of the vector space of inner derivations.
It is known that there exists a surjection from the vector space of derivations of the small category to the first Baues--Wirsching cohomology whose kernel is the vector space of inner derivations.
\end{abstract}

%


\maketitle

\section{Introduction}
\label{sec:Intro}

Baues and Wirsching \cite{MR814176} introduced a cohomology of a small category, which is called nowadays the Baues--Wirsching cohomology.
It is known that the Baues--Wirsching cohomology is a generalization of some cohomologies; e.g., the cohomology of a group $G$ with coefficients in a left $G$-module, the singular cohomology of the classifying space of a small category with coefficients in a field, and so on. 
Let $k$ be a field and $D$ a natural system on a small category $\C$; that is, a functor from the category of factorizations in $\C$ to the category $k$-$\Mod$ of left $k$-modules.
The $n$-th Baues--Wirsching cohomology of $\C$ with coefficients in $D$ is denoted by $\HBW^n(\C,D)$.
For an equivalence $\phi : \C \to \C'$ of small categories and a natural system $D$ on $\C$, Baues and Wirsching showed that the $k$-liner map $\tilde{\phi} : \HBW^n(\C,D) \to \HBW^n(\C',\phi^\ast D)$ induced by $\phi$ is an isomorphism for $n \in \ZZ$. 
The Baues--Wirsching cohomology is an invariant for the equivalence of small categories in this sense.

Assume that $\C$ is freely generated by a quiver and that $D = \check{D} \circ t$ is the composition of $\check{D}$ and the target functor $t$.
In this case, it is known that $\HBW^n(\C,D)$ vanishes for $n \ge 2$ and that $\HBW^0(\C,D)$ is isomorphic to the limit $\lim_{\C}\check{D}$.
Therefore, we focus on the first cohomology $\HBW^1(\C,D)$.
%
Let $k\C$ be the category algebra of $\C$, i.e. the algebra whose basis is a morphism of $\C$ and whose multiplication is the composition of morphisms (if the morphisms are not composable, then the multiplication is zero).
Since $\C$ is generated by $Q$, the category algebra is the path algebra $kQ$.
Define the functor $\pi_{\C}$ from $k\C$-$\Mod$ to the category $k$-$\Mod^{\C}$ of functors from $\C$ to $k$-$\Mod$ as follows:
$\pi_{\C}$ maps an object $M$ in $k\C$-$\Mod$ to the functor which maps $x \in \Ob(\C)$ to $\id_x \cdot M$ and which maps $u \in \Mor(\C)$ to the left multiplicative map of $u$; and 
$\pi_{\C}$ maps a morphism $f$ in $k\C$-$\Mod$ to the natural transformation $\{ f|_{\id_x \cdot M} \}_{x \in \Ob(\C)}$.
Since the set of objects in $\C$ is finite, $\pi_{\C}$ is an equivalence of categories. (See \cite{MR0294454}.)
Our algorithm introduced in this article computes the first cohomology $\HBW^1(\C,\pi_{\C}(N) \circ t)$ for a left $k\C$-module $N$.

The authors give a description of the first Baues--Wirsching cohomology in the case where $\C$ is a $B_2$-free poset \cite{Momose--Numata}.
The algorithm in this paper is a generalization of the idea of the special case.

This article is organized as follows: 
In Section \ref{subsec:1st BW}, we define some notation.
In Section \ref{subsec:algorithm}, we give algorithms.
In Section \ref{sec:Main results}, we show our main result.
We calculate the first Baues--Wirsching cohomology 
for some examples in Section \ref{sec:example}.

\section{Definition}
\label{sec:Def}

\subsection{Definition of the first Baues--Wirsching cohomology}
\label{subsec:1st BW}

We define some notation on the first Baues--Wirsching cohomology in this section.

Let $P$ and $Q$ be finite sets, $s$ and $t$ maps from $Q$ to $P$.
We call the set $Q$ equipped with the triple $(P;s,t)$ a \emph{finite quiver}.
We call an element of $P$ a \emph{vertex} and call an element of $Q$ an \emph{arrow}.
An arrow $f \in Q$ such that $s(f) = a$ and $t(f) = b$ is denoted by $f : a \to b$.
We call a sequence $f_{1} \cdots f_{l}$ of arrows a \emph{path of length $l$} if $s(f_{i}) = t(f_{i+1})$ for all $i$.
A path $f_{1} \cdots f_{l}$ such that $t(f_1) = s(f_{l})$ is called a \emph{cycle}.
We say that a quiver $Q$ is \emph{acyclic} if $Q$ has no cycle.
Let $Q'$ be a subset of $Q$ and $P'$ a subset of $P$.
We call the set $Q'$ equipped with the triple $(P';s|_{Q'},t|_{Q'})$ a \emph{subquiver} of $Q$ if $s(Q')$ and $t(Q')$ are subsets of $P'$.

Let $Q$ be a finite quiver.
The category defined in the following manner is called the \emph{small category freely generated by $Q$}:
\begin{itemize}
\item the set of objects is the set of vertices of $Q$; 
\item a morphism from $x$ to $y$ is a path from $x$ to $y$; 
\item the identity $\id_x$ is the path from $x$ to $x$ of length $0$; and 
\item if $s(f) = t(g)$, then the composition of morphisms $f$ and $g$ is the concatenation of paths $f$ and $g$.
\end{itemize}

Let $\C$ be a small category freely generated by $Q$.
The category $\F(\C)$ defined in the following manner is called the \emph{category of factorizations in $\C$}:
\begin{itemize}
\item the objects are morphisms in $\C$; 
\item a morphism from $\alpha$ to $\beta$ is a pair $(u,v)$ of morphisms in $\C$ such that $\beta = u \circ \alpha \circ v$; and 
\item the composition of $(u',v')$ and $(u,v)$ is defined by $(u',v') \circ (u,v) = (u' \circ u , v \circ v')$.
\end{itemize}
A covariant functor from $\F(\C)$ to $k$-$\Mod$ is called a \emph{natural system} on a small category $\C$.
Let $D$ be a natural system on the small category $\C$.
For $\alpha \in \Ob(\F(\C))$, $D_{\alpha}$ denotes the $k$-module corresponding to $\alpha$.
For a pair $(u,v)$ of composable morphisms, we define $u_{\ast}$ and $v^{\ast}$ by 
\begin{align*}
u_{\ast}=D(u , \id_{s(v)}) : D_{v} \to D_{u \circ v},\\ v^{\ast}=D(\id_{t(u)} , v) : D_{u} \to D_{u \circ v}.
\end{align*}
Let $d : \Mor(\C) \to \prod_{\varphi \in \Mor(\C)} D_{\varphi}$ be a map such that $d(f) \in D_{f}$ for each $f \in \Mor(\C)$.
We call $d$ a \emph{derivation} from $\C$ to $D$ if $d(f \circ g) = f_{\ast}(dg) + g^{\ast}(df)$ for each pair $(f,g)$ of composable morphisms.
We define $\Der(\C,D)$ to be the $k$-vector space of derivations from $\C$ to $D$.
We call $d$ an \emph{inner derivation} from $\C$ to $D$ if there exists an element $(n_x)_{x \in \Ob(\C)} \in \prod_{x \in \Ob(\C)}D_{\id_x}$ such that 
$d(f) = f_{\ast}(n_{s(f)}) - f^{\ast}(n_{t(f)})$ for each $f \in \Mor(\C)$.
We define $\Ider(\C,D)$ to be the $k$-vector space of inner derivations from $\C$ to $D$.
%
The first Baues--Wirsching cohomology $\HBW^1(\C,D)$ is the quotient space $\Der(\C,D)/\Ider(\C,D)$.

\begin{rem}
Let $Q$ be a quiver, $\C$ a small category freely generated by $Q$, $N$ a $k\C$-module, $t$ the target functor, and $\tilde{D}$ the natural system $\pi_{\C}(N) \circ t$.
For a pair $(u,v)$ of composable morphisms, 
$u_{\ast}$ (\emph{resp.} $v^{\ast}$) maps $m \in \tilde{D}_{v} = \id_{t(v)} \cdot N$ (\emph{resp.} $n \in \tilde{D}_{u} = \id_{t(u)} \cdot N$) to $ u\cdot m \in \tilde{D}_{u \circ v} = \id_{t(u)} \cdot N$ (\emph{resp.} $n \in \tilde{D}_{u \circ v} = \id_{t(u)} \cdot N$).
\end{rem}

\subsection{Definition of algorithms}
\label{subsec:algorithm}

In this section, we give algorithms to obtain generators of $\Ider(\C,D)$.

Let $Q$ be a finite quiver, and $P$ the set of vertices of $Q$.
For subsets $Q_1$, $Q_3$ of $Q$ and a subset $\hat{P}$ of $P$, we define the set $H(\hat{P}; Q, Q_1, Q_3)$ to be 
\begin{align*}
\Set{h \in Q_3 | \begin{array}{l} \text{$t(h) \in \hat{P}$.}\\ \text{$hp$ is not a cycle in $Q$ for any path $p$ in $Q_1$.} \end{array}}.
\end{align*}
For subsets $Q_1$, $Q_2$ of $Q$ and $h \in H(\hat{P}; Q, Q_1, Q_3)$, we define the set $G(Q_1, Q_2; h)$ to be 
\begin{align*}
\Set{g \in Q_2 | \begin{array}{l} \text{There exists a cycle in $Q_1 \cup Q_2 \cup \Set{h}$}\\ \text{which  contains $g$ and $h$.} \end{array}}.
\end{align*}

\begin{algorithm}
\label{algorithm:A}
\makebox{}
\begin{description}
 \item[Input] a finite quiver $Q$.
 \item[Output] $((a_i)_{i=1}^{l}; (b_i)_{i=1}^{m}; (f_1)_{i=1}^{l}; (g_i)_{i=1}^{n}; (h_i)_{i=1}^{r})$.
 \item[Procedure]\makebox{}
 \begin{enumerate}
  \item Let $P$ be the set of vertices of $Q$.
  \item Let $\check{P}=\emptyset$, $\hat{P}=P$, $Q_1=\emptyset$, $Q_2=\emptyset$, $Q_3=Q$.
  \item \label{procedure:h} While $H(\hat{P}; Q, Q_1, Q_3) \neq \emptyset$, do the following:
  \begin{enumerate}
  \item Choose an element $h \in H(\hat{P}; Q, Q_1, Q_3)$.
  \item Let $Q' = ((Q_1 \cup Q_2) \setminus G(Q_1, Q_2; h)) \cup \Set{h}$.
  \item Let $\bar{Q}$ be a maximal acyclic subquiver of $Q$ including $Q'$.
  \item Let $\check{P} = \Set{a \in P | \text{$\exists f \in \bar{Q}$ such that $t(f)=a$.}}$.
  \item Let $\hat{P} = P \setminus \check{P}$.
  \item For each $a \in \check{P}$, choose $f_a \in \bar{Q}$ so that $t(f_a) = a$.
  \item Let $Q_1 = \Set{f_a | a \in \check{P}}$, $Q_2 = Q' \setminus Q_1$, and $Q_3 = Q \setminus Q'$.
  \end{enumerate}
  \item \label{procedure:poset} Let $l = |\check{P}|$.
For $i = 1, \ldots , l$, do the following:
  \begin{enumerate}
   \item Choose a vertex $x \in \check{P}$ such that there exists no arrow in $Q_1$ whose source is $x$.
   \item Let $a_i = x$.
   \item For $\alpha \in Q_1$ so that $t(\alpha) = x$, let $f_{i} = \alpha$.
   \item Let $\check{P} = \check{P} \setminus \Set{x}$, and $Q_1 = Q_1 \setminus \Set{\alpha}$.
  \end{enumerate}
  \item Let $\Set{b_1, \ldots , b_m} = \hat{P}$.
  \item Let $\Set{g_1, \ldots , g_n} = Q_2$.
  \item Let $\Set{h_1, \ldots , h_r} = Q_3$.
\end{enumerate}
\end{description}
\end{algorithm}

\begin{rem}
In Step \ref{procedure:h} in Algorithm \ref{algorithm:A}, $|H(\hat{P}; Q, Q_1, Q_3)|$ strictly decreases since $|\hat{P}|$ decreases in each step.
Hence Step \ref{procedure:h} is a finite procedure.
\end{rem}

\begin{rem}
Let $((a_i)_{i=1}^{l}; (b_i)_{i=1}^{m}; (f_1)_{i=1}^{l}; (g_i)_{i=1}^{n}; (h_i)_{i=1}^{r})$ be an output of Algorithm \ref{algorithm:A}.
Let 
\begin{align*}
\check{P} &= \Set{a_1, \ldots , a_l},\\ 
\hat{P} &= \Set{b_1, \ldots , b_m},\\ 
Q_1 &= \Set{f_1, \ldots , f_l},\\ 
Q_2 &= \Set{g_1, \ldots , g_n}, \text{and}\\ 
Q_3 &= \Set{h_1, \ldots , h_r}.
\end{align*}
The set $\check{P} \coprod \hat{P}$ is decomposition of $P$.
The set $Q_1 \coprod Q_2 \coprod Q_3$ is also decomposition of $Q$.
By Step \ref{procedure:poset} in Algorithm \ref{algorithm:A}, $a_i$ corresponds to the target of $f_i$ for $i = 1, \ldots , l$.
Hence if there exists a path from $a_j$ to $a_i$ or a path from $b_j$ to $a_i$ in $Q_1$, then the path is unique.
Since the quiver $Q_1 \cup Q_2$ is a maximal acyclic subquiver of $Q$, we can regard $\check{P}$ as a poset.
Moreover, if $a_j \le a_i$ in the poset $\check{P}$, then the inequality $i \le j$ holds.
If $Q$ is a finite acyclic quiver, then $Q_3$ is the empty set.
By Step \ref{procedure:h} in Algorithm \ref{algorithm:A}, for $h_i$ so that $t(h_i) \in \hat{P}$, there exists a path $p$ in $Q_1$ such that $h_i p$ is a cycle in $Q$.
\end{rem}

\begin{algorithm}
\label{algorithm:B}
\makebox{}
\begin{description}
 \item[Input] $((a_i)_{i=1}^{l}; (b_i)_{i=1}^{m}; (f_1)_{i=1}^{l}; (g_i)_{i=1}^{n}; (h_i)_{i=1}^{r})$.
 \item[Output] $(V, W)$.
 \item[Procedure]\makebox{}
 \begin{enumerate}
  \item Let $Q_1 = \Set{f_1, \ldots , f_l}$.
  \item (We define elements $v_{i,j}$ in the path algebra $kQ$.) For $j=1, \ldots , l$, do the following:
  \begin{enumerate}
   \item For $i=1, \ldots , l$, let $v_{i,j} = 0$.
   \item Let $v_{j,j} = \id_{a_j}$
   \item For $i=1, \ldots , n$, do the following:
  \begin{enumerate}
   \item Let $v_{l+i,j} = 0$.
   \item If there exists a path $p$ from $a_j$ to $t(g_i)$ in $Q_1$, then let $v_{l+i,j} = v_{l+i,j} + p$.
   \item If there exists a path $p$ from $a_j$ to $s(g_i)$ in $Q_1$, then let $v_{l+i,j} = v_{l+i,j} - g_{i}p$.
  \end{enumerate}
   \item For $i=1, \ldots , r$, do the following:
  \begin{enumerate}
   \item Let $v_{l+n+i,j} = 0$.
   \item If there exists a path $p$ from $a_j$ to $t(h_i)$ in $Q_1$, then let $v_{l+n+i,j} = v_{l+n+i,j} + p$.
   \item If there exists a path $p$ from $a_j$ to $s(h_i)$ in $Q_1$, then let $v_{l+n+i,j} = v_{l+n+i,j} - h_{i}p$.
  \end{enumerate}
  \end{enumerate}
  \item Let $V = (v_{i,j})_{1\le i \le l+n+r,\ 1\le j \le l}$.
  \item (We define elements $w_{i,j}$ in the path algebra $kQ$.) For $j=1, \ldots , m$, do the following:
  \begin{enumerate}
   \item For $i=1, \ldots , l$, let $w_{i,j} = 0$.
   \item For $i=1, \ldots , n$, do the following:
  \begin{enumerate}
   \item Let $w_{l+i,j} = 0$.
   \item If there exists a path $p$ from $b_j$ to $t(g_i)$ in $Q_1$, then let $w_{l+i,j} = w_{l+i,j} + p$.
   \item If there exists a path $p$ from $b_j$ to $s(g_i)$ in $Q_1$, then let $w_{l+i,j} = w_{l+i,j} - g_{i}p$.
  \end{enumerate}
   \item For $i=1, \ldots , r$, do the following:
  \begin{enumerate}
   \item Let $w_{l+n+i,j} = 0$.
   \item If there exists a path $p$ from $b_j$ to $t(h_i)$ in $Q_1$, then let $w_{l+n+i,j} = w_{l+n+i,j} + p$.
   \item If there exists a path $p$ from $b_j$ to $s(h_i)$ in $Q_1$, then let $w_{l+n+i,j} = w_{l+n+i,j} - h_{i}p$.
  \end{enumerate}
  \end{enumerate}
  \item Let $W = (w_{i,j})_{1\le i \le l+n+r,\ 1\le j \le m}$.
\end{enumerate}
\end{description}
\end{algorithm}

\begin{rem}
Let $(V,W)$ be the output of Algorithm \ref{algorithm:B} for some input.
The matrix $(v_{i,j})_{1\le i \le l,\ 1\le j \le l}$ is the identity matrix, i.e., the diagonal matrix whose entries one $(\id_{a_1}, \ldots , \id_{a_l})$.
The matrix $(w_{i,j})_{1\le i \le l,\ 1\le j \le m}$ is the zero matrix.
\end{rem}


\section{Our main result}
\label{sec:Main results}

We show our main result in this section.
Our main result computes the first Baues--Wirsching cohomology via the column echelon matrix obtained by our algorithm.


Let $Q$ be a finite quiver, $\C$ a small category freely generated by $Q$.
Fix a left $k\C$-module $N$, and 
consider the natural system $\tilde{D} = \pi_{\C}(N) \circ t$.

Let $T = ((a_i)_{i=1}^{l}; (b_i)_{i=1}^{m}; (f_1)_{i=1}^{l}; (g_i)_{i=1}^{n}; (h_i)_{i=1}^{r})$ be the output of Algorithm \ref{algorithm:A} for $Q$.
We define the $k$-vector space $A_1$, $A_2$, and $A_3$ by 
\begin{align*}
A_1 & = \bigoplus_{i=1}^{l}\tilde{D}_{f_i},&
A_2 & = \bigoplus_{i=1}^{n}\tilde{D}_{g_i},\text{ and}&
A_3 & = \bigoplus_{i=1}^{r}\tilde{D}_{h_i}.
\end{align*}
Let $(V, W)$ be the output of Algorithm \ref{algorithm:B} for $T$. 
Let $v_j$ and $w_j$ be the $j$-th column vector of $V$ and $W$, respectively.
The vectors $v_j$ and $w_j$ are elements of $\bigoplus_{i=1}^{l+n+r}k\C$.
We define the $k$-vector spaces $\bar{V}$ and $\bar{W}$ by 
\begin{align*}
\bar{V} & = \Braket{ v_{j}n_{a_j} | n_{a_j} \in \id_{a_j}\cdot N, 1 \le j \le l },\\
\bar{W} & = \Braket{ w_{j}n_{b_j} | n_{b_j} \in \id_{b_j}\cdot N, 1 \le j \le m }.
\end{align*}

\begin{thm}
\label{main_thm}
The first Baues--Wirsching cohomology $\HBW^1(\C,\tilde{D})$ is isomorphic to 
\begin{align*}
(A_1 \oplus A_2 \oplus A_3)/(\bar{V}+\bar{W})
\end{align*}
as $k$-vector spaces.
\end{thm}



\begin{proof}
According to Baues and Wirsching \cite{MR814176}, 
if $\C$ is freely generated by $S \subset \Mor(\C)$, then 
we can identify $\Der(\C,D)$ with $\prod_{\alpha \in S}D_{\alpha}$.
Via the identification, $\Ider(\C,D)$ is 
the $k$-vector space
\begin{align*}
\left\{(\alpha_{\ast}(n_{s(\alpha)}) - \alpha^{\ast}(n_{t(\alpha)}))_{\alpha} \in \prod_{\alpha \in S}D_{\alpha} \Bigg| (n_x)_{x} \in \prod_{x \in \Ob(\C)}D_{\id_x} \right\}.
\end{align*}
Let 
\begin{align*}
Q &=\Set{f_i | 1 \le i \le l} \cup \Set{g_i | 1 \le i \le n} \cup \Set{h_i | 1 \le i \le r}, \text{ and} \\
P &=\Set{a_i | 1 \le i \le l} \cup \Set{b_i | 1 \le i \le m}.
\end{align*}
It follows that $\Der(\C,\tilde{D}) \cong A_1 \oplus A_2 \oplus A_3$.
Hence $\Ider(\C,\tilde{D})$ is isomorphic to the $k$-vector space
\begin{align*}
B = \Set{ (\alpha n_{s(\alpha)} - n_{t(\alpha)})_{\alpha \in Q} \in A | (n_x)_{x} \in \bigoplus_{x \in P}\id_{x}\cdot N }.
\end{align*}
For $x \in P$ and $m \in \id_{x}\cdot N$, we define $r_{x}m = (r_{x,\alpha}m)_{\alpha \in Q} \in A$ by 
\begin{align*}
r_{x,\alpha}m & = 
\begin{cases}
-\alpha m & (s(\alpha) = x)\\
m & (t(\alpha) = x)\\
0 & (\text{otherwise}).
\end{cases}
\end{align*}
It is clear that the $k$-vector space $B$ is equal to 
\begin{align*}
\Braket{r_{x}m | x \in P, m \in \id_{x}\cdot N }.
\end{align*}

For $j = 1, \ldots , l$ and $n_{a_j} \in \id_{a_j}\cdot N$, we define $\overline{r_{a_j}}n_{a_j}$ to be $r_{a_j}n_{a_j} + \sum_{k=1}^{i-1}\overline{r_{a_k}}f_{k}n_{a_j}$.
For $j = 1, \ldots , m$ and $n_{b_j} \in \id_{b_j}\cdot N$, we define $\overline{r_{b_j}}n_{b_j}$ to be $r_{b_j}n_{b_j} + \sum_{k=1}^{l}\overline{r_{a_k}}f_{k}n_{b_j}$.
It follows from the direct calculation that 
$\overline{r_{a_j}}n_{a_j}$ and $\overline{r_{b_j}}n_{b_j}$ are equal to $v_{j}n_{a_j}$ and $w_{j}n_{b_j}$, respectively.
Hence we have Theorem \ref{main_thm}.
\end{proof}

\section{Some examples}
\label{sec:example}

In this section, we apply our algorithm to some examples of finite quivers to calculate the first Baues--Wirsching cohomology.
First we apply our algorithm to some quivers whose set of vertices is a $B_2$-free poset, which is discussed in \cite{Momose--Numata}.

\begin{exmp}
\label{chain}
Let $P_{n} = \Set{1, \ldots , n}$.
Define $\alpha_{i}$ to be an arrow from $i+1$ to $i$.
Let $Q_{n} = \Set{\alpha_i | i = 1, \ldots , n-1}$.
The quiver $Q_n$ is a chain in Figure \ref{fig:1}.
\begin{figure}
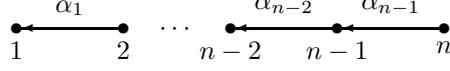

\centering
\figI
\caption{The quiver in Example \ref{chain}.}
\label{fig:1}
\end{figure}
An output of Algorithm \ref{algorithm:A} for $Q_n$ is 
\begin{align*}
a_i &= i \text{ for } i = 1, \ldots , n-1,\\ 
b_1 &= n, \text{ and}\\ 
f_i &= \alpha_i \text{ for } i = 1, \ldots , n-1.
\end{align*}
Consider the small category $\C_{n}$ generated by $Q_{n}$.
An output of Algorithm \ref{algorithm:B} is 
\begin{align*}
v_{j} &= \bigoplus_{k=1}^{n-1} \delta_{j,k}\id_{a_k} \text{ for } j = 1, \ldots , n-1, \text{ and}\\ 
w_{1} &= 0^{\oplus (n-1)}.
\end{align*}
For a $k\C_{n}$-module $N$, 
\begin{align*}
A_1 &= \bigoplus_{k=1}^{n-1}\id_{a_{k}}\cdot N,\\ 
A_2 &= 0,\\ 
A_3 &= 0,\\ 
\bar{V}\ &= \bigoplus_{k=1}^{n-1}\id_{a_{k}}\cdot N, \text{ and}\\ 
\bar{W} &= 0.
\end{align*}
By Theorem \ref{main_thm}, we have 
\begin{align*}
\HBW^1(\C_{n},\pi_{\C_{n}}(N) \circ t) = 0.
\end{align*}
\end{exmp}

\begin{exmp}
\label{->a<-}
Let $P_{n} = \Set{1, \ldots , n} \cup \Set{x=0}$.
Define $\alpha_{i}$ to be an arrow from $i$ to $x$.
Let $Q_{n} = \Set{ \alpha_i | i = 1, \ldots , n}$.
The quiver $Q_n$ is a quiver such that the targets of each arrow is $x$.
See Figure \ref{fig:3}.
\begin{figure}
\centering
\figIII
\caption{The quiver in Example \ref{->a<-}.}
\label{fig:3}
\end{figure}
An output of Algorithm \ref{algorithm:A} for $Q_n$ is 
\begin{align*}
a_1 &= x,\\ 
b_i &= i \text{ for } i = 1, \ldots , n,\\ 
f_1 &= \alpha_n, \text{ and}\\ 
g_i &= \alpha_i \text{ for } i = 1, \ldots , n-1.
\end{align*}
Consider the small category $\C_{n}$ generated by $Q_{n}$.
An output of Algorithm \ref{algorithm:B} is 
\begin{align*}
v_{1} &= (\id_{a_1})^{\oplus n},\\ 
w_{j} &= 0 \oplus \left(\bigoplus_{k=1}^{n-1} (-\delta_{j,k}g_{k})\right) \text{ for } j = 1, \ldots , n, \text{ and}\\
w_{n} &= 0 \oplus \left(f_{1}^{\oplus (n-1)}\right).
\end{align*}
For a $k\C_{n}$-module $N$, 
\begin{align*}
A_1 &= \id_{a_{1}}\cdot N,\\ 
A_2 &= \bigoplus_{k=1}^{n-1}\id_{a_{1}}\cdot N,\\ 
A_3 &= 0,\\ 
\bar{V}\ &= \Braket{m^{\oplus n} | m \in \id_{a_1}\cdot N}, \text{ and}\\ 
\bar{W} &= \left(\bigoplus_{k=1}^{n-1}g_{k}\cdot N \right) + \Braket{f_{1}^{\oplus (n-1)}n_{b_3} | n_{b_3} \in \id_{b_3}\cdot N}.
\end{align*}
By Theorem \ref{main_thm}, we have 
\begin{align*}
&\HBW^1(\C,\pi_{\C_{n}}(N) \circ t) \cong (A_1 \oplus A_2)/(\bar{V}+\bar{W}).
%
\end{align*}
Moreover, if $N = k\C_{n}$, then 
\begin{align*}
&\HBW^1(\C_{n},\pi_{\C_{n}}(k\C_{n}) \circ t)\\ \cong 
&\left(\bigoplus_{k=1}^{n-1}(\Braket{ \id_{a_1}, f_{1}}+\Braket{ g_{j} | j \neq k })\right) \bigm/ \Braket{f_{1}^{\oplus(n-1)} }.
\end{align*}
\end{exmp}

\begin{exmp}
\label{zigzag_loop}
Let $P_{n} = \Set{x_{j} = (0,j) | j \in \ZZ/n\ZZ} \cup \Set{y_{j} = (1,j) | j \in \ZZ/n\ZZ}$.
Define $\alpha_j$ and $\beta_j$ to be arrows from $y_j$ and $y_{j-1}$ to $x_j$, respectively.
Let $Q_{n}=\Set{ \alpha_j | j \in \ZZ/n\ZZ} \cup \Set{ \beta_j | j \in \ZZ/n\ZZ}$. 
The quiver $Q_n$ is a zigzag circle in Figure \ref{fig:4}.
\begin{figure}
\centering
\figIV
\caption{The quiver in Example \ref{zigzag_loop}.}
\label{fig:4}
\end{figure}
An output of Algorithm \ref{algorithm:A} for $Q_n$ is 
\begin{align*}
a_j &= x_j \text{ for } j = 1, \ldots , n,\\ 
b_j &= y_j \text{ for } j = 1, \ldots , n,\\ 
f_j &= \alpha_j \text{ for } j = 1, \ldots , n, \text{ and}\\ 
g_j &= \beta_j \text{ for } j = 1, \ldots , n.
\end{align*}
Consider the small category $\C_{n}$ generated by $Q_{n}$. 
An output of Algorithm \ref{algorithm:B} is 
\begin{align*}
v_{j} &= (\bigoplus_{k=1}^{n} \delta_{j,k}\id_{a_k}) \oplus (\bigoplus_{k=1}^{n} \delta_{j,k}\id_{a_k})
 \text{ for } j = 1, \ldots , n, \text{ and}\\ 
w_{j} &= (0^{\oplus n}) \oplus (\bigoplus_{k=1}^{n} (\delta_{j,k}f_{k}-\delta_{j+1,k}g_{k})) \text{ for } j = 1, \ldots , n.
\end{align*}
For a $k\C_{n}$-module $N$, 
\begin{align*}
A_1 &= \bigoplus_{k=1}^{n}\id_{a_{k}}\cdot N
,\\ 
A_2 &= \bigoplus_{k=1}^{n}\id_{a_{k}}\cdot N
,\\ 
A_3 &= 0,\\ 
\bar{V}\ & = \Braket{ v_{j}n_{a_j} | n_{a_j} \in \id_{a_j}\cdot N, j = 1, \ldots , n }, \text{ and}\\ 
\bar{W} & = \Braket{ w_{j}n_{b_j} | n_{b_j} \in \id_{b_j}\cdot N, j = 1, \ldots , n }.
\end{align*}
By Theorem \ref{main_thm}, we have 
\begin{align*}
&\HBW^1(\C_{n},\pi_{\C_{n}}(N) \circ t) 
\cong 
(A_1 \oplus A_2)/(\bar{V}+\bar{W}).
\end{align*}
Moreover, if $N = k\C_{n}$, then 
\begin{align*}
\HBW^1(\C_{n},\pi_{\C_{n}}(k\C_{n}) \circ t) \cong 
\bigoplus_{k=0}^{n-1}\Braket{ \id_{a_k}, f_{k} }.
\end{align*}
\end{exmp}

Next we consider examples which are not posets.

\begin{exmp}
\label{circle_1}
Let $P_{n} = {\ZZ/n\ZZ}$.
Define $\alpha_j$ to be an arrow from $j+1$ to $j$.
Let $Q_{n}=\Set{\alpha_j | j \in \ZZ/n\ZZ}$.
The quiver $Q_n$ is a circle in Figure \ref{fig:6}.
\begin{figure}
\centering
\figVI
\caption{The quiver in Example \ref{circle_1}.}
\label{fig:6}
\end{figure}
An output of Algorithm \ref{algorithm:A} for $Q_n$ is 
\begin{align*}
a_j &= j \text{ for } j = 1, \ldots , n-1,\\ 
b_1 &= n,\\ 
f_j &= \alpha_j \text{ for } j = 1, \ldots , n-1, \text{ and}\\ 
h_1 &= \alpha_n.
\end{align*}
Consider the small category $\C_{n}$ generated by $Q_{n}$.
An output of Algorithm \ref{algorithm:B} is 
\begin{align*}
v_{j} &= \left(\bigoplus_{k=1}^{n-1} \delta_{j,k}\id_{a_k}\right)\oplus (-h_{1}f_{1}\cdots f_{j-1}) \text{ for } j = 1, \ldots , n,\\ 
w_{1} &= 0^{\oplus (n-1)}\oplus (\id_{b_{1}} - h_{1}f_{1}\cdots f_{n-1}).
\end{align*}
For a $k\C_{n}$-module $N$, 
\begin{align*}
A_1 &= \bigoplus_{k=1}^{n-1}\id_{{a_k}}\cdot N ,\\ 
A_2 &= 0,\\ 
A_3 &= \id_{{b_1}}\cdot N ,\\ 
\bar{V}\ & = \Braket{ v_{j}n_{a_j} | n_{a_j} \in \id_{a_j}\cdot N, j = 1, \ldots , n-1 }, \text{ and}\\ 
\bar{W} & = \Braket{ w_{1}n_{b_{1}} | n_{b_{1}} \in \id_{b_{1}}\cdot N }.
\end{align*}
By Theorem \ref{main_thm}, we have 
\begin{align*}
&\HBW^1(\C_{n},\pi_{\C_{n}}(N) \circ t) 
\cong 
(A_1 \oplus A_3)/(\bar{V}+\bar{W}). 
\end{align*}
Moreover, if $N = k\C_{n}$, then 
\begin{align*}
&
\HBW^1(\C_{n},\pi_{\C_{n}}(k\C_{n}) \circ t)\\ 
\cong
&
\Braket{ \id_{a_0} } + \Braket{ h_{1}f_{1}\cdots f_{j-1} | j = 1, \ldots , n-1 }.
\end{align*}
\end{exmp}

\begin{exmp}
\label{circle_2}
Let $P_{n} = {\ZZ/n\ZZ}$.
Define $\alpha_j$ to be an arrow from $j+1$ to $j$, and $\beta_j$ to be an arrow from $j$ to $j+1$.
Let $Q_{n}=\Set{\alpha_j | j \in \ZZ/n\ZZ} \cup \Set{\beta_j | j \in \ZZ/n\ZZ}$.
The quiver $Q_n$ is a circle in Figure \ref{fig:7}.
\begin{figure}
\centering
\figVII
\caption{The quiver in Example \ref{circle_2}.}
\label{fig:7}
\end{figure}
An output of Algorithm \ref{algorithm:A} for $Q_n$ is 
\begin{align*}
a_j &= j \text{ for } j = 1, \ldots , n-1,\\ 
b_1 &= n,\\ 
f_j &= \alpha_j \text{ for } j = 1, \ldots , n-1,\\ 
g_1 &= \beta_n,\\ 
h_j &= \beta_j \text{ for } j = 1, \ldots , n-1, \text{ and}\\ 
h_n &= \alpha_n.
\end{align*}
Consider the small category $\C_{n}$ generated by $Q_{n}$. 
We define $p_{i,j}$ in $k\C$ by 
\begin{align*}
p_{i,j} & = 
\begin{cases}
f_{i}\cdots f_{j} & (\text{if $i<j+1$})\\
\id_{a_{i}} & (\text{if $i=j+1$})\\
0 & (\text{if $i>j+1$})
\end{cases}.
\end{align*}
An output of Algorithm \ref{algorithm:B} is 
\begin{align*}
v_{i,j} &= \delta_{i,j}\id_{a_j}\\
& \text{ for } i = 1, \ldots , n-1,\ j = 1, \ldots , n-1,\\ 
v_{n,j} &= p_{1,j-1} \\
& \text{ for } j = 1, \ldots , n-1,\\ 
v_{n+i,j} &= p_{i+1,j-1}-h_{i}p_{i,j-1}\\
& \text{ for } i = 1, \ldots , n-1,\ j = 1, \ldots , n-1,\\ 
v_{2n,j} &= -h_{n}p_{1,j-1}\\ 
& \text{ for } j = 1, \ldots , n-1,\\ 
w_{i,1} &= 0\\
& \text{ for } i = 1, \ldots , n-1,\\ 
w_{n,1} &= p_{1,n-1} - g_{1} ,\\ 
w_{n+i,1} &= p_{i+1,n-1}-h_{i}p_{i,n-1}\\
& \text{ for } i = 1, \ldots , n-1, \text{ and}\\ 
w_{2n,1} &= \id_{b_1}-h_{n}p_{1,n-1}.
\end{align*}
Let $v_{j} = \bigoplus_{i=1}^{2n}v_{i,j}$ for $j = 1, \ldots , n-1$, and $w_{1} = \bigoplus_{i=1}^{2n}w_{i,1}$.
For a $k\C_{n}$-module $N$, 
\begin{align*}
A_1 &= \bigoplus_{k=1}^{n-1}\id_{a_k}\cdot N ,\\ 
A_2 &= \id_{a_1}\cdot N ,\\ 
A_3 &= (\bigoplus_{k=2}^{n-1}\id_{a_k}\cdot N) \oplus (\id_{b_1}\cdot N) \oplus (\id_{b_1}\cdot N) ,\\ 
\bar{V}\ & = \Braket{ v_{j}n_{a_j} | n_{a_j} \in \id_{a_j}\cdot N, j = 1, \ldots , n-1 }, \text{ and}\\ 
\bar{W} & = \Braket{ w_{1}n_{b_1} | n_{b_1} \in \id_{b_1}\cdot N }.
%
%
%
\end{align*}
By Theorem \ref{main_thm}, we have 
\begin{align*}
&\HBW^1(\C_{n},\pi_{\C_{n}}(N) \circ t) 
\cong 
(A_1 \oplus A_2 \oplus A_3)/(\bar{V}+\bar{W}).
\end{align*}
\end{exmp}



%

%
%

\end{document}